\documentclass[11pt]{article}
\usepackage{amsfonts}
\usepackage{lipsum}
\usepackage{latexsym}
\usepackage{amsmath}
\usepackage{amssymb}
\usepackage{amsthm}
\usepackage{graphicx}
\usepackage{fullpage}
\usepackage{enumerate}
\usepackage{color}
\usepackage{bbm}
\newtheorem{theorem}{Theorem}[section]
\newtheorem{lemma}[theorem]{Lemma}

\newtheorem{proposition}[theorem]{Proposition}

\newtheorem{remark}[theorem]{Remark}
\newtheorem{conjecture}[theorem]{Conjecture}

\newtheorem{problem}[theorem]{Problem}
\theoremstyle{definition}

\newtheorem{thmy}{Theorem}

\newtheorem*{note*}{Note}

\newcommand{\R}{\mathbb R}
\newcommand{\Rn}{{\mathbb R}^n}

\newcommand{\Sn}{\mathbb{ S}^{n-1}}

\newcommand{\Ha}{{{\cal H}^{n-1}}}

\makeatother

\newcommand\blfootnote[1]{%
  \begingroup
  \renewcommand\thefootnote{}\footnote{#1}%
  \addtocounter{footnote}{-1}%
  \endgroup
}

\begin{document}


\title{\bf A non-existence result for the $L_p$-Minkowski problem}
\date{}
\medskip
\author{Christos Saroglou}
\maketitle
\begin{abstract}
We show that given a real number $p<1$, a positive integer $n$ and a proper subspace $H$ of $\Rn$, the measure on the Euclidean sphere $\Sn$, which is concentrated in $H$ and whose restriction to the class of Borel subsets of $\Sn\cap H$ equals the spherical Lebesgue measure on $\Sn\cap H$, is not the $L_p$-surface area measure of any convex body.  This, in particular, disproves a conjecture from [Bianchi, B\"or\"oczky, Colesanti, Yang, The $L_p$-Minkowski problem for $-n<p<1$, Adv. Math. (2019)].
\end{abstract}
\section{Introduction}
\hspace*{1.5em}Given a convex body $K$ in $\Rn$ and $p\in\R$, the $L_p$-surface area measure $S_p(K,\cdot)$ of $K$ is a Borel measure on the Euclidean sphere $\Sn$, defined by
$$dS_p(K,\cdot):=h_K^{1-p}dS(K,\cdot),$$
where $h_K$ is the support function of $K$ and $S(K,\cdot)$ is the surface area measure of $K$ (see Section 2 for definitions). The $L_p$-Minkowski problem (as of the existence) asks the following. 
\begin{problem}\label{problem}
Let $\mu$ be a given Borel measure on $\Sn$. Find necessary and sufficient conditions for $\mu$ to be the $L_p$-surface area measure of some convex body in $\Rn$.
\end{problem}
The case $p=1$ is the classical Minkowski problem (see Section 2). The problem for general $p$ was initiated by Lutwak \cite{lutwak_1993}, where it was solved in the even case, for $p\in(1,n)\cup(n,\infty)$. Problem \ref{problem} was subsequently studied by several authors and it is still of current interest. The $L_p$-Minkowski problem for $p>1$ in the non-symmetric case was treated in \cite{hug_lutwak_yang_zhang_2005}. Chou and Wang \cite{chou_wang_2006} studied the problem for $p\geq -n-1$, in the case where $\mu$ is absolutely continuous with respect to the spherical Lebesgue measure (and some extra regularity is assumed). A necessary and sufficient condition in the even case for $p=0$ was found in \cite{boroczky_lutwak_yang_zhang_2013}, while a sufficient condition was obtained in \cite{chen_li_zhu_2019} in the non-symmetric case, for $p=0$. For existence results in the case $p<0$, see \cite{bianchi_2019}, \cite{zhu_2017} and the references therein (see, also, \cite{haberl_lyz}, \cite{huang_lyz}, \cite{gauss} for other important variants of the Minkowski problem).

In the case $0<p<1$, it was shown in \cite{chen_li_zhu_2017} that if $\mu$ is a finite Borel measure on $\Sn$ that is not concentrated in a great subsphere, then there exists a convex body whose $L_p$-surface area measure equals $\mu$. Note that, as the example of a simplex with one vertex at the origin shows, $S_p(K,\cdot) $ (for $p<1$) may be concentrated in a proper subspace of $\Rn$ and, therefore, the condition mentioned previously is not necessary. A more general result (still in the case $0<p<1$) appeared in  \cite{bianchi_2019}, where the following conjecture was proposed.
\begin{conjecture}\label{conjecture} (Bianchi, Boroczky, Colesanti, Yang \cite{bianchi_2019})
Let
$0<p<1$ and $\mu$ be a finite Borel measure on $\Sn$ that is not concentrated in a pair of antipodal points. Then, there exists a convex body $K$ in $\Rn$, such that $$S_p(K,\cdot)=\mu.$$
\end{conjecture}
Conjecture \ref{conjecture} was confirmed in the plane independently by \cite{trinh} and \cite{chen_li_zhu_2017}.
The main purpose of this note is to show that Conjecture \ref{conjecture} fails if $n\ge 3$. In fact, we prove a slightly more general result. Namely, for any $1\leq k\leq n-1$, there exists a finite Borel measure on $\Sn$, such that the linear span of its support is $k$-dimensional but the measure fails to be the $L_p$-surface area measure of any convex body (notice that by the result mentioned previously, the case $k=1$ is already known, while for $k=n$ it is impossible to find such a measure).  Moreover, our examples are valid for all $p<1$; other non-existence results, for $p\leq -n,$ appear for instance in \cite{chou_wang_2006} and \cite{du}. 

\begin{theorem}\label{thm-main}
Let $p<1$ and $H$ be a subspace of $\Rn$ of dimension $1\leq k\leq n-1$. If $\mu_H$ is the measure on $\Sn$ which is concentrated in $H$ and whose restriction to the family of Borel sets in $\Sn\cap H$ equals the $(k-1)$-dimensional Hausdorff measure on $\Sn\cap H$, then there exists no convex body $K$ in $\Rn$ satisfying 
$$S_p(K,\cdot)=\mu_H.$$
\end{theorem}
We mention that it is well known that for $p\geq1$, $S_p(K,\cdot)$ cannot be concentrated in a proper subspace of $\Rn$. Therefore, Theorem \ref{thm-main} trivially holds for $p\geq 1$ as well.

Although we are unable to formulate a general conjecture, it seems plausible to us that Conjecture \ref{conjecture} should hold if the measure $\mu$ is additionally assumed to have finite support.

The proof of Theorem \ref{thm-main} will be given in Section 3. In Section 2, we fix some notation and collect some facts concerning convex sets, that will be used in the proof. 
\section{Background and notation}
\hspace*{1.5em}We denote by $\langle \cdot , \cdot\rangle$ the standard inner product in $\Rn$ and the origin by $o$. 
The Euclidean norm of a vector $x\in\Rn$ is denoted by $|x|$. We set $B_2^n=\{x\in\Rn:|x|\le 1\}$ for the unit ball in $\Rn$ and $\Sn=\{x\in\Rn:|x|=1\}$ for its unit sphere. For $v\in \Sn$ and $t\in\R$, set
$$H^-(v,t):=\{x\in\Rn: \langle x,v\rangle\leq t\}.$$
Let $A\subseteq \Rn$.
The notation $A|E$ stands for the orthogonal projection of $A$ onto a subspace $E$. We write $\partial A$ for the boundary of $A$. Set, also, $A^\perp:=\{x\in\Rn:\langle x,y\rangle=0,\ \forall y\in A\}$ for its orthogonal complement and $u^\perp:=\{u\}^\perp$, if $u\in\Rn$. Finally, for $a\ge0$, the $a$-dimensional Hausdorff measure will be denoted by ${\cal H}^a(\cdot)$.

Below, we collect some basic facts from convex geometry that will be needed in the proof of Theorem \ref{thm-main}. For more information, we refer to the books of Schneider \cite{schneider_2013} and Gardner \cite{gardner_2006}. 
A convex set $C$ is called a {\it convex cone}, if $rx\in C$, for all $r\geq 0$ and for all $x\in C$. Clearly, a set $C$ is a convex cone if and only if $ax+\beta y\in C$, for all $a,\beta\geq 0$ and for all $x,y\in C$. Notice that according to this definition, any non-empty convex cone contains the origin.

We say that a vector $u\in\Sn$ {\it supports} a convex set $C\subseteq \Rn$ at $w\in\partial C$, if $\langle u,x\rangle\leq \langle u,w\rangle$, for all $x\in C$. If $C$ happens to be a convex cone and $u\in\Sn$ supports $C$ at some boundary point of $C$, then it is clear that $\langle u,x\rangle\leq 0$, for all $x\in C$. That is, $u$ supports $C$ at $o$ (and $o\in\partial C$). In this case, we will briefly say that $u$ supports $C$.

Let $K$ be a {\it convex body} in $\Rn$ (that is, convex, compact, with non-empty interior). The {\it support function} $h_K$ of $K$ is defined by
$$h_K(x):=\sup\{\langle x,y\rangle:x\in K\},\qquad x\in\Rn.$$
It is clear by the definition that $h_K$ is convex and positively homogeneous, that is $h_K(tx)=th_K(x)$, for $t\ge 0$. Moreover, $h_K(o)=0$ and $h_K$ is non-negative (resp. strictly positive) on $\Rn\setminus\{o\}$ if and only if $K$ contains the origin (resp. $K$ contains the origin in its interior). The support function $h_{K|E}:E\to \R$ of the orthogonal projection of $K$ onto a subspace $E$ can be computed by
$$h_{K|E}=(h_K)|_E.$$

The {\it surface area measure} $S(K,\cdot)$ of $K$ is a (necessarily finite) Borel measure on $\Sn$, given by
$$S(K,\omega)=\Ha(\{x\in\partial K: \exists v\in\omega,\textnormal{ such that } \langle x,v\rangle=h_K(v)\}),\qquad\omega\textnormal{ is a Borel set in }\Sn.$$
Minkowski's Existence and Uniqueness Theorem states that a finite Borel measure $\mu$ is the surface area measure of a convex body $K$ in $\Rn$ if and only if $\mu$ is not concentrated in a proper subspace of $\Rn$ and the barycentre of $\mu$ is at the origin. Moreover, $K$ is unique up to translation.
\section{Proof of Theorem \ref{thm-main}}
\hspace*{1.5em}We start the proof of Theorem \ref{thm-main} with the following observation.
\begin{lemma}\label{l->0}
Let $K$ be a convex body in $\Rn$ that contains the origin and assume that for some $u\in\Sn$ it holds $h_K(u)=0$. If for some $q<0$, the function $v\mapsto \langle v,u\rangle h_K(v)^q$ is integrable on $\Sn$, then
$$\int_{\Sn} \langle v,u\rangle h_K(v)^q \, d\Ha(v)>0.$$
\end{lemma}
\begin{proof} For $t\geq 0$, define the hyperplane $H(t):=u^\perp+tu$. Set, also, $H^+:=\{x\in\Rn:\langle x,u\rangle\geq 0\}$ and $H^-:=\{x\in\Rn:\langle x,u\rangle\leq 0\}$.

Using the fact that the function $v\mapsto \langle v,u\rangle h_K(v)^q$ is integrable on $\Sn$ and integrating in polar coordinates, yields
\begin{eqnarray*}
\int_{B_2^n\cap H^\pm}\langle x,u\rangle h_K(x)^qe^{-1/|x|}\, dx&=&\int_{\Sn\cap H^\pm}\int_0^1r^{n-1}\langle rv,u\rangle h_K(rv)^qe^{-1/|rv|}\, dr\, d\Ha(v)\\
&=&\int_0^1r^{n+q}e^{-1/r}\, dr\cdot\int_{\Sn\cap H^\pm}\langle v,u\rangle h_K(v)^q\, d\Ha(v).
\end{eqnarray*}
Thus, $\int_{B_2^n\cap H^\pm}\pm\langle x,u\rangle h_K(x)^qe^{-1/|x|}\, dx<\infty$ and we need to prove that 
\begin{equation}\label{eq-l->0-1}
\int_{B_2^n\cap H^+}\langle x,u\rangle h_K(x)^qe^{-1/|x|}\, dx>-\int_{B_2^n\cap H^-}\langle x,u\rangle h_K(x)^qe^{-1/|x|}\, dx.
\end{equation}

Let $z\in u^\perp$ and $t>0$. Observe that
\begin{equation}\label{eq-for-remark}h_K(z+tu)\leq h_K(z-tu)+2th_K(u)=h_K(z-tu)\end{equation}
and, hence,
\begin{equation*}
h_K(z+tu)^q e^{-1/|z+tu|}\geq h_K(z-tu)e^{-1/|z-tu|}.
\end{equation*}
This shows that for $0<t<1$, it holds
\begin{eqnarray}\label{eq-l-general-metamain}\int_{[B_2^n\cap H(t)]-tu}h_K(z+tu)^qe^{-1/|z+tu|}\, dz&\geq& \int_{[B_2^n\cap H(t)]-tu}h_K(z-tu)^qe^{-1/|z-tu|}\, dz\nonumber\\
&=&\int_{[B_2^n\cap H(-t)]+tu}h_K(z-tu)^qe^{-1/|z-tu|}\, dx.
\end{eqnarray}
Let $t\in(0,1)$ be such that the function $z\mapsto h_K(z+tu)^qe^{-1/|z+tu|}$ is integrable on $[B_2^n\cap H(t)]-tu$. Then it is clear that equality holds in (\ref{eq-l-general-metamain}) if and only if equality holds in (\ref{eq-for-remark}), for all $z\in [B_2^n\cap H(t)]-tu$.

Notice that (\ref{eq-l-general-metamain}) can be equivalently written as
\begin{equation}\label{eq-l_>0-2}
\int_{B_2^n\cap H(t)}h_K(y)^qe^{-1/|y|}\, dy\geq \int_{B_2^n\cap H(-t)}h_K(y)^qe^{-1/|y|}\, dy,
\end{equation}
for all $0<t<1$. 
Therefore, by Fubini's Theorem, we obtain
\begin{eqnarray*}
\int_{B_2^n\cap H^+} \langle x,u\rangle h_K(x)^qe^{-1/|x|}\, dx&=&\int_0^1t\int_{B_2^n\cap H(t)}h_K(y)^qe^{-1/|y|}\, dy\, dt\\
&\geq& \int_0^1 t\int_{B_2^n\cap H(-t)}h_K(y)^qe^{-1/|y|}\, dy\, dt\\
&=& \int_{B_2^n\cap H^-}-\langle x,u\rangle h_K(x)^qe^{-1/|y|}\, dx.
\end{eqnarray*}Assume, now, that equality holds in the previous inequality. Then equality holds in (\ref{eq-l_>0-2}) (and therefore in (\ref{eq-l-general-metamain})) for almost every $t\in(0,1)$. Consequently, equality must hold in (\ref{eq-for-remark}) for almost every $t\in(0,1)$ and for all $z\in[B_2^n\cap H(t)]-tu $. Given this and the continuity of $h_K$, we conclude that $h_K(-u)=h_K(u)=0$, which is a contradiction because $K$ is assumed to have non-empty interior. The validity of (\ref{eq-l->0-1}) follows.
\end{proof}
We will need the following statement concerning convex cones.
\begin{proposition}\label{prop}

Let $R$ be closed convex cone in $\Rn$, which is supported by at least one unit vector contained in the linear span of $R$ (that is, $R$ is not a subspace of $\Rn$). Then, there exists $u_0\in\Sn\cap (-R)$, such that $u_0$ supports $R$.
\end{proposition}
\begin{proof}
We may, clearly, assume that $R$ is full dimensional. Furthermore, we may assume by a standard approximation argument that if $x,y$ are boundary points of $R$, then $ax+\beta y$ is an interior point of $R$, provided that $a^2+\beta^2>0$. Set
$$A:=\sup_{(u,v)\in (\Sn)^2}\{\langle -u,v\rangle:u\textnormal{ supports }R\textnormal{ and }v\in R\}.$$
It follows by compactness that there exists a pair $(u_0,v_0)\in(\Sn)^2$, such that $u_0$ supports $R$, $v_0\in R$ and $\langle -u_o,v_0\rangle=A$.
We need to prove that $A=1$. Let us assume that $A<1$ instead.

{\it Claim 1.} There exists a unique vector $v_1\in u_0^\perp\cap R\cap \Sn$. \\
{\it Proof.} First assume that there exist two distinct vectors $v_1,v_1'\in u_0^\perp\cap R\cap \Sn$. Then, $v_1+v_1'\in u_0^\perp\cap R$. Since $u_0^\perp$ is a supporting hyperplane of $R$, it follows that $u_0^\perp\cap R\subseteq \partial R$, hence all $v_1$, $v_1'$, $v_1+v_1'$ are contained in $\partial R$, which contradicts the previous assumption. This shows that if there exists such $v_1$, then this must be unique.

Let us show the existence of $v_1$. If $u_0^\perp\cap R\cap \Sn=\emptyset$, then by compactness, there exists $0<c<1$, such that $\langle-u_0,v\rangle>c$, for all $v\in\Sn\cap R$. Thus, there exists $t_0>0$, such that for all $t\in(-t_0,t_0)$, for all $v\in\Sn\cap R$ and for all $u\in \Sn$, it holds $$\langle -(u_0-tu),v\rangle>0.$$Consequently, if $u\in \Sn$, then for all $t\in(-t_0,t_0)$, the unit vector $(u_0-tu)/|u_0-tu|$ supports $R$. Thus, 
$$\left(\frac{\langle -u_0+tu,v_0\rangle}{|-u_0+tu|}\right)'_{t=0}=0.$$Since
\begin{equation}\label{eq-prop-1}
\left(\frac{\langle -u_0+tu,v_0\rangle}{|-u_0+tu|}\right)'_{t=0}=\langle u,v_0\rangle-\langle u_0,v_0\rangle\langle u,u_0\rangle
\end{equation}
(this holds even if we do not assume $u$ to be a unit vector), we conclude that 
$$\langle u,v_0\rangle=\langle u_0,v_0\rangle\langle u,u_0\rangle,$$for all $u\in\Sn$. This shows that $v_0=\pm u_0$, which is a contradiction, because we assumed that $A<1$ and because $u_0\not\in R$. Our claim follows. $\square$

{\it Claim 2.} $\langle v_0,v_1\rangle\geq 0$.\\
{\it Proof.} For $t>0$, $v_0+tv_1\in R$, thus
$$\left(\frac{\langle -u_0,v_0+tv_1\rangle}{|v_0+tv_1|}\right)'_{t=0^+}\leq 0.$$Eqs (\ref{eq-prop-1}), then, shows that 
$$\langle -u_0,v_1\rangle-\langle-v_0,-u_0\rangle\langle v_1,-v_0\rangle\leq 0$$
or equivalently (since $\langle u_0,v_1\rangle=0$ and $\langle u_0,v_0\rangle<0$), $\langle v_0,v_1\rangle\geq 0$. $\square$

To finish with our proof set
$$u_s:=\frac{u_0-s(v_0+v_1)}{|u_0-s(v_0+v_1)|}.$$
We claim that $u_s$ supports $R$, if $s>0$ and $s$ is sufficiently small. To see this, notice that since $u_0\not\in R$, it holds $u_0-s(v_0+v_1)\neq o$, for all $s>0$. Moreover, since $v_0\neq-v_1$, we have $\langle v_1,v_0+v_1\rangle>0$. By continuity, there exists $t>0$, such that $\langle v, v_0+v_1\rangle>0$, for all $v\in\Sn\cap R$ with $|v-v_1|<t$. This shows that $$\langle -(u_0-s(v_0+v_1)),v\rangle\geq s\langle v,v_0+v_1\rangle>0,$$ for all $s>0$ 
and for all $v\in\Sn\cap R$ with $|v-v_1|<t$. On the other hand, since $\langle -u_0,v\rangle>0$, for all $u\in(\Sn\cap R)\setminus\{v_1\}$, there exists $s_0>0$, such that 
$$\langle -u_0+s(v_0+v_1),v\rangle>0,$$for all $v\in \Sn\cap R$ with $|v-v_1|\geq t$ and for all $0<s<s_0$. Consequently, for $0<s<s_0$ and for $v\in\Sn\cap R$, it holds $\langle -(u_0-s(v_0+v_1)),v\rangle\geq 0$. That is, $u_s$ supports $R$. 

As a consequence, we have
$$(\langle -u_s,v_0\rangle)'_{s=0^+}\leq 0,$$which by (\ref{eq-prop-1}) gives
$$\langle v_0+v_1,v_0\rangle-\langle u_0,v_0\rangle \langle v_0+v_1, u_0\rangle\leq 0.$$
Since $\langle v_1,u_0\rangle=0$ and since by Claim 2, $\langle v_0+v_1,v_0\rangle\geq |v_0|^2=1$, we find
$$1-\langle u_0,v_0\rangle^2\leq 0,$$which is impossible by the assumption $A<1$ and the fact that $u_0\not\in R$. This completes our proof.
\end{proof}
The following well known fact will also be needed.
\begin{lemma}\label{l-sch}
For any convex body $K$ in $\R^n$, it holds
$$K=\bigcap_{v\in supp\, S(K,\cdot)}H^-(v,h_k(v)).$$
\end{lemma}
\begin{proof} The proof follows (for instance) from \cite[Theorem 2.2.6 and Lemma 4.5.2 ]{schneider_2013}.
\end{proof}
Below, we state and prove a consequence of Lemma \ref{l-sch} and Proposition \ref{prop}.
\begin{lemma}\label{l-general}
Let $p<1$, $K$ be a convex body in $\Rn$ that contains the origin and $H$ be a proper subspace of $\Rn$. Set $S:=supp\, S(K,\cdot)\setminus H$ and $T:=\{v\in\Sn:h_K(v)=0\}$. If $S_p(K,\cdot)$ is concentrated in $H$, then there exist $u_0\in \Sn\cap H$ and $y\in H^\perp$, satisfying the following.
\begin{enumerate}
    \item[\textnormal{(i)}] $h_{K|H}(u_0)=0$.
    \item[\textnormal{(ii)}] $\langle u_0,v\rangle\geq 0$, for all $v\in T\cap H$.
    \item[\textnormal{(iii)}] $\langle u_0+y,v\rangle\geq 0$, for all $v\in S$.
\end{enumerate}
\end{lemma}
\begin{proof}
Notice that $S\neq\emptyset$, since $S(K,\cdot)$ cannot be concentrated in $H$. 
We first claim that $S\subseteq T$. That is, $h_K(v)=0$, for all $v\in S$ (in particular, this shows that $o\in\partial K$). 
Indeed, if this is not the case, then $h_K(v)>0$, for some $v\in S$. By continuity, there exists an open set $\Omega$ in $\Sn$, such that $(h_K)|\Omega $ is bounded away from zero. In fact, we can choose $\Omega$ to be contained in the open set $\Sn\setminus H$. Then, by the definition of the support of a measure, we have $S(K,\Omega)>0$, thus
$$S_p(K,\Omega)=\int_\Omega h_K(v)^{1-p}\, dS(K,v)>0,$$ 
which contradicts the fact that $S_p(K,\cdot)$ is concentrated in $H$.

Next, define the closed convex cone $$C:=\bigcap_{v\in T}H^-(v,0).$$Notice that by Lemma \ref{l-sch}, it holds
\begin{eqnarray}
K&=&\left(\bigcap_{v\in S}H^-(v,h_K(v))\right)\cap \left(\bigcap_{v\in \Sn\cap H}H^-(v,h_K(v))\right)\nonumber\\
&=&\left(\bigcap_{v\in S}H^-(v,0)\right)\cap \left(\bigcap_{v\in \Sn\cap H}H^-(v,h_K(v))\right)\nonumber\\
&\supseteq&\left(\bigcap_{v\in T}H^-(v,0)\right)\cap \left(\bigcap_{v\in \Sn\cap H}H^-(v,h_K(v))\right)\nonumber\\
&=&C\cap  \left(\bigcap_{v\in \Sn\cap H}H^-(v,h_K(v))\right)=:C\cap P.\nonumber
\end{eqnarray}
Since the reverse inclusion trivially holds true, we conclude that 
\begin{eqnarray}\label{eq-l-general}
K=C\cap P.
\end{eqnarray}
We claim that 
\begin{equation}\label{eq-disjoint}
H^\perp\cap C=\{o\}.    
\end{equation}
To see this, observe that since $C$ is a convex cone, for $y\in H^\perp\setminus\{o\}$, only the following cases are possible.
\begin{enumerate}[(a)]
\item $\{ry:r\geq 0\}\subseteq C$.
\item $\{ry:r\leq 0\}\subseteq C$.
\item $\R y\cap C=\{o\}$.
\end{enumerate}
If any of the two first cases occurs, since $\R y$ is (trivially) contained in the (unbounded) cylinder $P$, then by (\ref{eq-l-general}), either $\{ry:r\geq 0\}\subseteq K$ or $\{ry:r\leq 0\}\subseteq K$. This contradicts the fact that $K$ is bounded. Thus, we are left with case (c) and (\ref{eq-disjoint}) is proved.

Eqs (\ref{eq-disjoint}) and the Hahn-Banach Theorem show easily that there exists a supporting hyperplane $E$ (necessarily a subspace of $\Rn$) of $C$, that contains $H^\perp$. 
Notice that if $E=u^\perp$ for some $u\in\Sn$, then $u$ must be contained in $H$. Thus, we have found a unit vector $u\in H$ that supports $C$. 

Since $\langle u,x\rangle \leq 0$, for all $x\in C$, it follows that $\langle u,x\rangle \leq 0$, for all $x\in C|H$. Furthermore, (\ref{eq-l-general}) shows that $C|H\supseteq K|H$ and, therefore, $C|H$ is a closed convex cone of dimension $k=\dim H$. Thus, by Proposition \ref{prop}, there exists a unit vector $u_0\in H$, such that $u_0$ supports $C|H$  and $-u_0\in C|H$. However, $o\in K|H\subseteq C|H$, which immediately shows that $h_{K|H}(u_0)=0$. Thus, $u_0$ satisfies assertion (i). Furthermore, observe that there exists $y\in H^\perp$, such that if we set $w:=-u_0-y$, then $w\in C$. By the definition of $C$, this is equivalent to \begin{equation}\label{eq-l-general-last}\langle w,v\rangle\leq 0, \qquad \forall v\in T\end{equation}and, since $S\subseteq T$, $u_0$ and $y$ satisfy (iii). The fact that $u_0$ satisfies (ii) follows trivially from (\ref{eq-l-general-last}). Our proof is complete. 
\end{proof}
\noindent
\begin{proof}[{\it Proof of Theorem \ref{thm-main}.}]
Assume that there exists such $K$. Let $S,\ T, \ u_0,\ y$ be as in the statement of Lemma \ref{l-general} and set $z:=u_0+y$. Then, ${\cal H}^{k-1}(T\cap H)=0$, otherwise the support of $\mu_H$ would be strictly contained in $\Sn\cap H$. By the assumption of Theorem \ref{thm-main}, for any continuous function $\varphi:(\Sn\cap H)\setminus T\to\R$, it holds 
\begin{eqnarray}\label{eq-proof-final}\int_{(\Sn\cap H)\setminus T}\varphi(v)\, dS(K,v)&=&\int_{(\Sn\cap H)\setminus T}\varphi(v)h_K(v)^{-(1-p)}d\, {\cal H}^{k-1}(v).\nonumber\\
&=&\int_{\Sn\cap H}\varphi(v)h_K(v)^{-(1-p)}d\, {\cal H}^{k-1}(v).\end{eqnarray}
In particular, since $S(K,\cdot)$ is finite, (\ref{eq-proof-final}) shows that the function $h_K^{-(1-p)}$ is integrable on $\Sn\cap H$ and, therefore, the function $v\mapsto \langle u_0,v\rangle h_K(v)^{-(1-p)}$ is also integrable on  $\Sn\cap H$.

Using (\ref{eq-proof-final}), Lemma \ref{l->0}, Lemma \ref{l-general} (i) and Lemma \ref{l-general} (ii), we deduce
\begin{eqnarray}\label{eq-proof-4}
\int_{\Sn\cap H}\langle u_0,v\rangle\, dS(K,v)&=&\int_{T\cap H}\langle u_0,v\rangle\, dS(K,v)+\int_{(\Sn\cap H)\setminus T}\langle u_0,v\rangle\, dS(K,v)\nonumber \\
&\geq& \int_{(\Sn\cap H)\setminus T}\langle u_0,v\rangle\, dS(K,v)\nonumber \\
&=&\int_{\Sn\cap H}\langle u_0,v\rangle h_K(v)^{-(1-p)}\, d{\cal H}^{k-1}(v)\nonumber\\
&=&\int_{\Sn\cap H}\langle u_0,v\rangle h_{K|H}(v)^{-(1-p)}\, d{\cal H}^{k-1}(v)>0.
\end{eqnarray}

Finally, Lemma \ref{l-general} (iii) and (\ref{eq-proof-4}) give
\begin{eqnarray*}
\int_{\Sn}\langle z,v\rangle\, dS(K,v)&=&\int_{S}\langle z,v\rangle\, dS(K,v)+\int_{\Sn\cap H}\langle z,v\rangle\, dS(K,v)\\
&\geq& \int_{\Sn\cap H}\langle z,v\rangle\, dS(K,v)\\
&=& \int_{\Sn\cap H}\langle u_0,v\rangle\, dS(K,v)>0.
\end{eqnarray*}
This violates Minkowski's Existence and Uniqueness Theorem (in particular the part that states that the barycentre of $S(K,\cdot)$ is at the origin) and, therefore, there cannot be any convex body $K$ satisfying $S_p(K,\cdot)=\mu_H$. \end{proof}
\begin{remark}
The strategy to prove Theorem \ref{thm-main} was to arrive at a contradiction as follows: 
If $K$ is a convex body such that $S_P(K,\cdot)=\mu_H$, then the barycentre of $S(K,\cdot)$ is not at the origin. The same reasoning can be applied to the equation \begin{equation}\label{eq-firey}h_K^{1-p}dS(K,\cdot)=dS_p(K,\cdot)=d\Ha(\cdot), \qquad p\neq 1, \end{equation}to show that if a convex body $K$ that contains $o$ satisfies (\ref{eq-firey}), then $K$ contains $o$ in its interior. This is, of course, an immediate consequence of Lemma \ref{l->0}, if $p<1$ and of the analogue of Lemma \ref{l->0} in the case $q>0$, if $p>1$. We mention that in \cite{brendle_choi_daskalopoulos_2017}, the authors proved that a strictly convex body with $C^\infty$-smooth boundary (due to a result of Caffarelli \cite{caffarelli_1991}, this turns out to be equivalent to the fact that $K$ contains the origin in its interior) satisfying (\ref{eq-firey}), for some $p>-n-1$, must be a ball. Clearly, the previous discussion shows that the regularity assumption in their result can be omitted. An extension of the result in \cite{brendle_choi_daskalopoulos_2017} appears in \cite{saroglou_2021}, where $h_K^{1-p}$ in (\ref{eq-firey}) is replaced with $1/G(h_K)$ and $G:(0,\infty)\to(0,\infty)$ is an appropriate function. A similar argument as above shows that the assumption ``$K$ contains $o$ in its interior'' in the main result of \cite{saroglou_2021} can be also omitted (after some obvious modifications), if $G$ is assumed to be strictly monotone.
\end{remark}
\begin{remark}
One can easily prove (without using Proposition \ref{prop}; see the proof of Lemma \ref{l-general}) that if $S_p(K,\cdot)$ is concentrated in a $k$-dimensional subspace $H$ of $\Rn$ (where $p<1$ and $1\leq k\leq n-1$), then $o\in \partial (K|H)$. Because of this and since the support function of any $k$-dimensional convex body that contains the origin in its boundary, raised to the power -$k$, is never integrable on $\mathbb{S}^{k-1}$, the proof of Theorem \ref{thm-main} becomes much easier (and the result is much more expected) if we restrict ourselves to the case $p\leq1-k$. 
\end{remark}

\vspace{1.5 cm}

\noindent Christos Saroglou \\
Department of Mathematics\\
University of Ioannina\\
Ioannina, Greece, 45110 \\
E-mail address: csaroglou@uoi.gr \ \& \ christos.saroglou@gmail.com

\end{document}